\documentclass[10pt]{amsart}
\usepackage{amscd,amsmath,amssymb,amsfonts}
\usepackage[all]{xy}
\usepackage{hyperref}
\usepackage{url}
\usepackage{stmaryrd}
\usepackage{color}

\theoremstyle{plain}
\newtheorem{thm}{Theorem}
\newtheorem{lem}[thm]{Lemma}
\newtheorem{cor}[thm]{Corollary}
\newtheorem{prop}[thm]{Proposition}

\newtheorem{defn}[thm]{Definition}
\newtheorem{conj}[thm]{Conjecture}

\theoremstyle{definition}

\newtheorem{question}[thm]{Question}
\newtheorem{nota}[thm]{Notations}

\newtheorem{rmk}[thm]{Remark}

\numberwithin{thm}{section}

\newcommand{\ga}[2]{\begin{gather}\label{#1}#2 \end{gather}}

\newcommand{\Spec}{{\rm Spec \,}}

\newcommand{\Gal}{{\rm Gal}}

\newcommand{\sE}{{\mathcal E}}

\newcommand{\sL}{{\mathcal L}}
\newcommand{\sM}{{\mathcal M}}

\newcommand{\sO}{{\mathcal O}}

\newcommand{\sZ}{{\mathcal Z}}
\newcommand{\A}{{\mathbb A}}

\newcommand{\C}{{\mathbb C}}

\newcommand{\F}{{\mathbb F}}
\newcommand{\G}{{\mathbb G}}

\newcommand{\N}{{\mathbb N}}

\newcommand{\Q}{{\mathbb Q}}

\newcommand{\Z}{{\mathbb Z}}

\DeclareMathOperator{\GL}{GL}
\DeclareMathOperator{\SL}{SL}
\DeclareMathOperator{\PGL}{PGL}

\begin{document}

\title[Integrality]{ Integrality of the Betti moduli space }
\author{Johan de Jong}
\author{H\'el\`ene Esnault }
\address{ Mathematics Hall, 2990 Broadway, New York, NY 10027, USA }
 \email{ dejong@math.columbia.edu}
\address{Freie Universit\"at Berlin, Arnimallee 3, 14195, Berlin,  Germany}
\email{esnault@math.fu-berlin.de}

\thanks{ During the preparation of this work, the second named  author was supported by the Samuel Eilenberg Chair of Columbia University. The excellent working conditions and the friendly atmosphere are gratefully acknowledged.}
\maketitle

\begin{abstract}
 If in a given rank $r$, there is an irreducible complex local system with
torsion determinant and quasi-unipotent monodromies at infinity on a smooth
quasi-projective variety, then for every prime number $\ell$, there is an
absolutely irreducible $\ell$-adic local system of the same rank, with the same
determinant and monodromies at infinity, up to semi-simplification. A finitely
presented group is said to be weakly integral with respect to a torsion
character and a rank $r$ if once there is an irreducible rank $r$ complex
linear representation, then for any $\ell$, there is an absolutely irreducible
one of rank $r$ and determinant this given character,  which is defined over
$\bar{ \mathbb{Z}}_\ell$. We prove that this property is a new obstruction for
a finitely presented group to be the fundamental group of a smooth
qusi-projective complex variety. The proofs rely on the arithmetic Langlands
program via the existence of Deligne's companions (L. Lafforgue, Drinfeld) and
the geometric Langlands program via de Jong's conjecture (Gaitsgory for $\ell
\ge 3$). We also define weakly arithmetic complex local systems and show they
are Zariski dense in the Betti moduli. Finally we show that our method gives an
arithmetic proof of  the  Corlette-T. Mochizuki theorem, proved using tame pure
imaginary harmonic metrics,  which shows  the pull-back by a morphism between two
smooth complex algebraic varieties of a semi-simple complex local system is
semi-simple.

\end{abstract}

\section{Introduction}

Let $X$ be a smooth quasi-projective variety defined over the field of complex
numbers. Let $r$ be a positive natural number. Let $\sL$ be a rank $1$
complex local system on $X$ of finite order.
We fix a  smooth projective compactification $X\hookrightarrow \bar X$ with boundary divisor $\bar X\setminus X$  being a strict normal crossings divisor, called in the sequel  ``good compactification''.
 For each
irreducible component $D_i$ of $\bar X\setminus X$, we fix a quasi-unipotent
conjugacy class $T_i \subset \GL_r(\C)$.
For any conjugacy class $T \subset \GL_r(K)$, where $K$ is an
algebraically closed field of characteristic $0$, we denote by
$T^{ss} \subset \GL_r(K)$ the conjugacy class of its semi-simplification.

\begin{thm}[Weak integrality property with respect to $(r, \sL, T_i)$]
\label{thm:top}
Assume there is an irreducible topological rank $r$ complex local system
$\mathbb L_\C$ with determinant $\sL$ and monodromies in $T_i$ at infinity.
Then for any prime number $\ell$, there is an $\ell$-adic local system
$\mathbb L_\ell$ which
\begin{itemize}
\item[1)] has rank $r$ and is irreducible over $\bar \Q_\ell$,
\item[2)] has determinant $\sL$,
\item[3)] has monodromies $T_{i, \ell}$
at infinity such that $T_i^{ss} = T_{i,\ell}^{ss}$.
\end{itemize}
\end{thm}
For example, if  $X$ is smooth projective and $\mathbb L_\C$  is an irreducible rank $r$ local system with  trivial determinant, then Theorem~\ref{thm:top}  says that for all $\ell$, there is an irreducible  rank $r$ $\ell$-adic local system with trivial determinant.

\medskip 

Forgetting the conditions at infinity, if we fix a finitely presented group $\Gamma$, a natural number $r\ge 1$, and a rank $1$ torsion complex character $\chi: \Gamma\to \C^\times$, we pose the 

\begin{defn}
\label{defn:weak}
$\Gamma$ has the weak integrality property with respect to $(r, \chi)$
if, assuming there is an irreducible representation
$\rho : \Gamma\to \GL_r(\C)$ with determinant $\chi$,
then for any prime number $\ell$, there is a representation
$\rho_\ell: \Gamma\to \GL_r(\bar \Z_\ell)$ with determinant
$\chi$ which is irreducible over $\bar \Q_\ell$. 
\end{defn}

This property depends only on the isomorphism class of $\Gamma$. 
In \cite{BBV22}, the authors study the  $\SL(2, \C)$-character variety of irreducible representations of a residually finite  group $\Gamma_0$ defined in \cite[Theorem~4]{DS05}, with $2$ generators  $\{a,b\}$ and one relation $b^2=a^2ba^{-2}.$   They in particular show that it is $0$-dimensional, defined over $\Q$, with only two conjugate points, which correspond to irreducible dense complex representations,  which are {\it not integral} at the prime $2$. See Section~\ref{BBV}. We conclude that $\Gamma_0$ does not have the weak integrality property with respect to $(2, \mathbb I)$ where $\mathbb I$ is the trivial character.

\begin{thm}[Obstruction]
\label{thm:obs}

If $X$ is a smooth complex quasi-projective variety, then $\Gamma = \pi_1(X(\C))$
satisfies the weak integrality property for any pair $(r, \chi)$.
In particular, the group $\Gamma_0$ above cannot be the topological
fundamental group of a smooth complex quasi-projective variety.
\end{thm}

Our  obstruction for an abstract finitely presented group $\Gamma$ to be the topological fundamental group of a smooth complex quasi-projective seems to be of a new kind.  We do not have to specify a finite set of conjugacy classes in   $\Gamma$  which would be the local monodromies at infinity. See Section~\ref{sec:obs} for the proof of Theorem~\ref{thm:obs}.

\medskip

 Ultimately, as we shall see below, the proof of Theorem~\ref{thm:top} relies on the arithmetic Langlands program as proven by L. Lafforgue \cite{Laf02}, on the existence of $\ell$-adic companions shown by him in dimension $1$ and Drinfeld \cite{Dri12}  in higher  dimension, and on de Jong's conjecture \cite{dJ01}, proved  for $\ell \ge 3$ by Gaitsgory \cite{Gai07}, using the geometric Langlands program. 

\medskip
We now explain in which framework Theorem~\ref{thm:top} is located and proven.
Let $S$ be an affine scheme of finite type  over $\Z$ with $\sO(S) \subset \C$,
such that a  good compactification $X \hookrightarrow \bar X$ and a given complex point $x \in X$
have a model over $S$.  This means that we have a relative good compactification
$X_S \hookrightarrow \bar X_S$ over $S$ such that $\bar X_S\setminus X_S$ is a relative normal crossings divisor, we have  an $S$-point $x_S$ of $X_S$
whose base change to $\Spec(\C)$ recovers $x \in X \hookrightarrow \bar X$.
We also assume the orders of $\sL$ and of the eigenvalues of the $T_i$
are invertible on $S$.
For any closed point $s \in |S|$ of residue field $\F_q$ of characteristic
$p>0$, with a $\bar \F_p$-point $\bar s$ above it, we denote by 
\ga{}{sp_{\C,\bar s}: \pi_1(X_\C, x_\C)\to \pi_1^t(X_{\bar s}, x_{\bar s}) \notag}
the continuous  surjective specialization  homomorphism to the tame fundamental group 
\cite[Expos\'e~XIII 2.10,
Corollaire~2.12]{SGA1}. Precomposing with the profinite completion homomorphism
\ga{}{  \pi_1(X(\C), x(\C))\to \pi_1(X_\C, x_\C)\notag}
from the topological fundamental group to the \'etale one  over $\C$  yields the homomorphism
\ga{}{sp^{\rm top} _{\C,\bar s}: \pi_1(X(\C), x(\C))\to \pi_1^t(X_{\bar s}, x_{\bar s}). \notag}
 This enables us to consider the topological pull-back $(sp^{\rm top} _{\C,\bar s})^*(\mathbb L_{\ell, \bar s})$ of any tame $\ell$-adic local system $\mathbb L_{\ell, \bar s}$ on $X_{\bar s}$, in particular of  those  $\mathbb L_{\ell, \bar s}$ which are arithmetic, that is defined over $X_{\F_{q'}}$ for a finite extension $\F_q\to \F_{q'}\subset \bar \F_p$, as in \cite[Section~3]{EG18}.  The topological pull-back $(sp^{\rm top} _{\C,\bar s})^*(\mathbb L_{\ell, \bar s})$ is defined over $\bar \Z_\ell$.
We prove

\begin{thm}
\label{thm:arithm}
If there is one (resp.\ infinitely many pairwise non-isomorphic)
irreducible topological rank $r$ complex local system (resp.\ systems)
$\mathbb L_\C$ with determinant $\sL$ and monodromies
in $T_i$ at infinity, then there is a non-empty open subscheme
$S^\circ\subset S$ such that for any two closed points
$s, s'\in |S|$ of residual characteristics $p\neq p'$ it holds:
\begin{itemize}
\item[1)] for any prime number $\ell\neq p$ there is one (resp.\ infinitely many pairwise non-isomorphic) arithmetic local system (resp.\ systems) $\mathbb L_{\ell, \bar s}$ on $X_{\bar s}$;
\item[2)] which has (resp.\ have)  determinant $\sL$,
with quasi-unipotent monodromies $T_{i,\ell, \bar s}$  at infinity such that $T_i^{ss}=T_{i,\ell, \bar s}^{ss}$;
 \item[3)]   which is (are) irreducible over $\bar \Q_\ell$;
 \item[4)] for $\ell=p$ there is one (resp.\ infinitely many pairwise non-isomorphic) arithmetic local system (resp.\ systems) $\mathbb L_{p, \bar s'}$ on $X_{\bar s'}$ with 2),3) where $\ell$ is replaced by $p$;
\item[5)] for any prime number $\ell$,  the topological pull-backs $(sp^{\rm top} _{\C,\bar s})^*(\mathbb L_{\ell, \bar s})$ (which in the resp.\ case are pairwise non-isomorphic) 
have properties 2) and 3) as topological local systems. 
\end{itemize}
\end{thm}
Theorem~\ref{thm:arithm} 5)  for the non-resp.\ case immediately  implies Theorem~\ref{thm:top} by considering the topological pull-backs  $(sp^{\rm top} _{\C,\bar s})^*(\mathbb L_{\ell, \bar s})$.

\medskip

In Theorem~\ref{thm:arithm} we can in addition  single out specific extra properties for $\mathbb L_\C$ and request that the  $(sp^{\rm top} _{\C,\bar s})^*(\mathbb L_{\ell, \bar s})$ keep the same properties.

\begin{thm} \label{thm:type}
\begin{itemize}
\item[i)]
If in Theorem~\ref{thm:arithm}, we assume that $\mathbb L_{\C}$
is cohomologically rigid, then  we can choose
$\mathbb L_{\ell, \bar s}$ such that the topological pull-back
$(sp^{\rm top} _{\C,\bar s})^*(\mathbb L_{\ell, \bar s})$
is cohomologically rigid.
\item[ii)] If in Theorem~\ref{thm:arithm} we assume that the Zariski closure
of the monodromy of $\mathbb L_\C$ contains  $\SL_{r}( \C)$, then
we can choose $\mathbb L_{\ell, \bar s}$ such that the topological pull-back
$(sp^{\rm top} _{\C,\bar s})^*(\mathbb L_{\ell, \bar s})$
has the same property. 
\end{itemize}
\end{thm}

Theorem~\ref{thm:type} i)   follows  from  \cite[Theorem~1.1]{EG18} which is Simpson's integrality conjecture for cohomologically rigid local systems.  Its  method  of proof, developed  with M. Groechenig, is the starting point of this article.  We comment on this in Remark~\ref{rmk:pigeon}. Theorem~\ref{thm:type} ii) answers positively a question by A. Landesman  asked to us after the second named author lectured on Theorem~\ref{thm:arithm} in Harvard in October  2022.

\medskip
 
We now describe the method of proof of the non-resp. case of Theorem~\ref{thm:arithm}, see Section~\ref{sec:top}.  We make use of the fact that the Betti moduli space 
 $M_B(X,r,\sL, T_i)$  parametrizing irreducible local systems of rank $r$  with prescribed  torsion determinant  $\sL$ and quasi-unipotent monodromies  at infinity  in $T_i$ 
   is of finite type over a number ring $\sO_K$, see Section~\ref{sec:Betti}. We denote by  $M_B(X,r,\sL, T^{ss}_i)$  the disjoint union of the finitely many $M_B(X,r,\sL, T_i)$ so the semi-simplification of $T_i$ is $T_i^{ss}$ for all $i$. 
   The existence of an irreducible rank $r$ complex local system $\mathbb L_\C$  is equivalent to $M_B(X,r,\sL, T_i)$ being dominant over $\sO_K$.  By generic smoothness,  the underlying reduced scheme $ M_B(X,r,\sL,T_i)_{red}\subset M_B(X,r,\sL, T_i)$ is smooth over $\sO_K$ on some non-trivial open subscheme. We pick  a closed point $z$ in this locus, with residue field $\F_{\ell^m}$ for $\ell \ge 3$. This defines an $\F_{\ell^m}$-local system   on   $X_\C$. By Grothendieck's theory of the specialization of the \'etale fundamental group of $X_\C$ to the tame one in characteristic $p>0$,  this $\F_{\ell^m}$-local system descends to the mod $p$ reduction $X_{\bar \F_p}$  for $p$ large prime to $\ell$, and as the monodromy is finite, to $X_{\F_q}$ where $q=p^s$ for some $s\in \N_{>0}$. 
The completion of  $M_B(X,r,\sL, T_i)$ at $z$ is identified  with Mazur's deformation space of $z$ keeping the same conditions $(\sL, T_i)$. We can then apply de Jong's conjecture
 to the effect that $z$   lifts to an $\ell$-adic  local system  $\mathbb L_{\ell, \bar \F_p} $ which is artihmetic   and  is still absolutely irreducible.
  We now apply the existence of companions for any $\ell'\neq p$. This yields arithmetic 
 $\ell'$-adic local systems  $\mathbb L_{\ell', \bar \F_p}$ on $X_{\bar \F_p}$ with the same determinant and monodromies at infinity, modulo semi-simplification. Pulling-back those to  $X(\C)$ yields a point in $M_B(X,r,\sL, T_i^{ss})(\bar \Z_{\ell'})$. 
  This proves the theorem for $\ell\neq p$. At $p$ we just choose a different specialization to $X_{\bar \F_{p'}}$ for $p\neq p'$.

 In the resp. case we do the same replacing $M_B(X,r,\sL, T_i)$ by one component, dominant over  $\sO_K$,  which over $\C$ contains infinitely many of them. 
 
 \medskip
For Theorem~\ref{thm:type} ii), we use in addition to Theorem~\ref{thm:arithm}  the fact that the existence of $\mathbb L_\C$ with Zariski dense monodromy implies that the locus in  $M_B(X,r, \sL, T_i)(\C)$ of complex points with Zariski dense monodromy is Zariski dense, see Section~\ref{sec:dens}.
 
\medskip

The method of proof described above invites us to define the notion of a   {\it weakly arithmetic complex  local system} $\mathbb L_\C$:  there is an identification of $\C$ with $\bar \Q_\ell$ such that the resulting topological local system $\mathbb L_{\bar \Q_\ell}$ defined over $\bar \Q_\ell$ is in fact $\ell$-adic and descends to an arithmetic  $\ell$-adic local system on some reduction $X_{\bar \F_p}$ mod $p$ (see Definition~\ref{defn:arithm}). The method of proof described above enables us to show that the weakly arithmetic complex local systems with a fixed determinant $\sL$ are dense in the Betti moduli space parametrizing irreducible complex local systems  $M_B(X,r,\sL)$ of rank $r$ and determinant $\sL$. In fact the proof does not use the companions in characteristic $p>0$, instead it uses the invariance of this locus by complex conjugation over $\C$, see Theorem~\ref{thm:dens}.

\medskip

 To summarize, in Section~\ref{sec:top} we prove  Theorem~\ref{thm:top} and Theorem~\ref{thm:arithm}, non-resp. case. In  Section~\ref{sc:dens} we define the notion of weakly arithmetic complex local systems and prove their density in the Betti moduli.  In Section~\ref{sec:obst} we prove Theorem~\ref{thm:obs}.  In Section~\ref{sec:resp} we prove  the resp. case of Theorem~\ref{thm:arithm}. In Section~\ref{sec:dens} we prove Theorem~\ref{thm:type} ii).  Finally in Section~\ref{sec:RQ} we make some comments, formulate some questions, and, as a curiosity, we give a proof  of the theorem by Corlette and T. Mochizuki that a morphism between normal complex varieties respects semi-simplicity of local systems. The original proof uses the harmonic theory. Our proof uses de Jong's conjecture (and not the companions).

\medskip 
\noindent
{\it Acknowledgements:}   The article makes use  of the companions over a finite field for a problem on complex local systems. This idea   has been developed in \cite{EG20}. We thank Michael Groechenig for the discussions we had at the time, which impacted a whole development afterwards. We thank Alexander Petrov for general enlightening discussions on his work and ours. We thank Mark Kisin, Aaron Landesman  and Will Sawin for interesting questions and answers to ours. We thank Emmanuel Breuillard for kindly  writing down for us the example documented in  Section~\ref{BBV}, which shows that our weak integrality property for a finitely presented group is indeed an obstruction for this group to come from algebraic geometry.   We  warmly thank the referee for a thorough, precise  and helpful report which helped us to improve  the presentation of our article and correct a mistake in the proof of Theorem~\ref{thm:Moc}.

\section{Proof of Theorems ~\ref{thm:top} and ~\ref{thm:arithm}, non-resp.\ case} \label{sec:top}

\subsection{The Betti moduli space} \label{sec:Betti}
Let $(X \hookrightarrow \bar X, r, D_i, T_i, \sL)$ be the  notation used in Theorem~\ref{thm:top}.
The datum $(X, r, \sL, T_i)$ defines a number field $K$
and its ring of integers $\sO_K$ over which $\sL$ and the
eigenvalues of the $T_i$ are defined. There is an algebraic stack
$\sM(X, r, \sL, T_i)$  of finite type over $\Spec(\sO_K)$
parametrizing irreducible local systems on $X$ of rank $r$,
with determinant $\sL$, and monodromies at infinity in $T_i$. 
See \cite[2.1]{Dri01} where it is denoted by ${\rm Irr}_r^X$, and
in \cite[Section~2]{EG18} where the determinant and the conditions
at infinity are taken into account, it is denoted by
$\underline{M}$.  See also the footnote \footnote{There is a typo defining the irreducibility condition,
which on p.~4282 should be tested not only on geometric generic points but
on all geometric points.}. We denote by $M_B(X, r, \sL, T_i)$
the associated coarse moduli space, and call it the Betti moduli space.
In the proof of the following lemma we will see that it exists and is
a separated scheme of finite type over $\Spec(\sO_K)$.

\begin{lem}
\label{lem:gerbe}
The morphism $\sM(X, r, \sL, T_i) \to M_B(X, r, \sL, T_i)$
exhibits the source as a $\G_m$-gerbe over the target.
\end{lem}

\begin{proof}
Most of the steps in this proof are justified in \cite{WE18} and \cite{EG18};
we only add arguments for the parts which are not shown there. Choose
generators $\gamma_1, \ldots, \gamma_n$ for the topological fundamental
group $\Gamma = \pi_1(X(\C), x(\C))$. Then there is a
locally closed (closed
if the conjugacy classes $T_i$ are semi-simple)  subscheme
$M^\square \subset (\GL_{r, \sO_K})^n$ with the following property:
for any $\sO_K$-algebra $R$, the $R$-points of $M^\square$ correspond
bijectively to homomorphisms $\rho : \Gamma \to \GL_r(R)$ whose associated
local system on $X$ defines an $R$-point of the stack $\sM(X, r, \sL, T_i)$. 
The correspondence sends $\rho$ to $(\rho(\gamma_1), \ldots, \rho(\gamma_n))$.
We obtain a morphism $M^\square \to \sM(X, r, \sL, T_i)$.
There is an action of the group scheme $G = \GL_{r, \sO_K}$ on $M^\square$
over $\Spec(\sO_K)$, and we have
$$
\sM(X, r, \sL, T_i) = [M^\square / G]
$$
as algebraic stacks. Since $M^\square$ is of finite type over
$\Spec(\sO_K)$ the same is true for $\sM(X, r, \sL, T_i)$. The action
of $G$ on $M^\square$ factors through an action of the group scheme
$\bar G = \PGL_{r, \sO_K}$ on $M^\square$ over $\Spec(\sO_K)$.
Since for every algebraically closed field $k$,
every representation $\rho : \Gamma \to \GL_r(k)$ corresponding to
a $k$-point of $M^\square$ is  by definition irreducible, we see that the action of
$\bar G$ on $M^\square$ is scheme theoretically free.
Hence the quotient
$$
M_B(X, r, \sL, T_i) = [M^\square / \bar G]
$$
is an algebraic space by
\cite[\href{https://stacks.math.columbia.edu/tag/06PH}{Tag 06PH}]{SP}.
Since $M^\square$ is of finite type over $\Spec(\sO_K)$ the same
is true for $M_B(X, r, \sL, T_i)$. For any scheme $T$ endowed
with an action of $\bar G$ the morphism $[T/G] \to [T/\bar G]$
is a $\mathbb G_m$-gerbe.
Although in the rest of the article we never use anything beyond
the facts already proven (at the cost of working with algebraic
spaces in addition to schemes), below we briefly indicate why
$M_B(X, r, \sL, T_i)$ is separated and a scheme. This is standard
but we have not been able to find a reference in the literature.

\medskip\noindent
To show that $M_B(X, r, \sL, T_i)$ is separated is equivalent
to proving that the action of $\bar G$ on $M^\square$ is closed, i.e.,
that the morphism
$\Psi : \bar G \times M^\square \to M^\square \times M^\square$
is a closed immersion. Freeness of the action means that $\Psi$ is
a monomorphism. By
\cite[\href{https://stacks.math.columbia.edu/tag/04XV}{Tag 04XV}]{SP}
it suffices to show that $\Psi$ is universally closed. To check this
in turn by
\cite[\href{https://stacks.math.columbia.edu/tag/04XV}{Tag 04XV}]{SP}
it suffices to check the existence part of the valuative criterion
for discrete valuation rings. Unwinding the definitions this boils
down to the following: given a discrete valuation ring $R$ with
uniformizer $\pi$, residue field $k$, and fraction field $L$,
given two homomorphisms $\rho_1, \rho_2 : \Gamma \to \GL_r(R)$ such that
$\bar \rho_1 : \Gamma \to \GL_r(k)$ is irreducible,
if $\rho_1$ and $\rho_2$ are isomorphic as representations over $L$,
then $\rho_1$ and $\rho_2$ are isomorphic as representations over $R$.
This follows from the fact that all $\Gamma$-$\rho_1$-invariant
lattices in $L^r$ are of the form $\pi^nR^r$,  using irreducibility and
integrality of $\rho_i$.

\medskip\noindent
By \cite[\href{https://stacks.math.columbia.edu/tag/03XX}{Tag 03XX}]{SP}
if we can construct a quasi-finite morphism $M_B(X, r, \sL, T_i) \to N$
to a scheme $N$, then $M_B(X, r, \sL, T_i)$ is a scheme. Let
$\Omega \subset \Gamma$ be a finite subset. Given
$\gamma \in \Omega$ we can associate to a representation $\rho$ of $\Gamma$
the characteristic polynomial of $\rho(\gamma)$. This defines a $G$-invariant
morphism $M^\square \to \prod_{\gamma \in \Omega} \A^r_{\sO_K}$
and hence a morphism
$$
M_B(X, r, \sL, T_i) \to \prod\nolimits_{\gamma \in \Omega} \A^r_{\sO_K}
$$
We claim  that if  $\Omega$ is large enough,  this morphism is quasi-finite onto its image 
 and the set of  field value points of its fibres is either empty or consists of one point. This  then
 finishes the proof. Namely, by the Brauer-Nesbit theorem the
isomorphism class of an irreducible representation $\rho$
over an algebraically closed field $k$
is determined by its character (see for example
\cite[Theorem~7.20]{Lam91}). On the other hand, since $\Gamma$ is finitely
generated, the pseudocharacter $\gamma \mapsto \det(T - \rho(\gamma))$
for any representation $\rho$ of fixed rank $r$ is determined by the
values on finitely many elements of $\Gamma$ for example
by \cite[Proposition~2.38]{Che14}.
\end{proof}

\subsection{Proof of Theorem~\ref{thm:arithm}, non-resp.~case}
\label{pf:non-resp}
Let $(S, X_S \hookrightarrow \bar X_S, x_S)$ be as in the introduction.
The existence of a point $\mathbb L_\C$ in $M_B(X,r,\sL,T_i)(\C)$
in the assumption of Theorem \ref{thm:arithm} tells us that
the structure morphism
\ga{}{ \epsilon: M_B(X,r,\sL,T_i)\to  {\rm Spec}(\sO_K) \notag}
is dominant. By generic smoothness, there is a non-empty open subscheme
$M^\circ \subset M_B(X, r, \sL, T_i)_{red}$ of the reduced scheme  such that
\ga{}{\epsilon|_{M^\circ}: M^\circ\to {\rm Spec}(\sO_K) \notag}
is smooth, dominant, and has values in $ {\rm Spec}(\sO_K) ^\circ$ where $ {\rm Spec}(\sO_K) ^\circ \to {\rm Spec}(\Z)$ is smooth over its image.
Let $z \in |M^\circ|$ be a closed point, so of residue field
$\F_{\ell^m}$ for a prime number $\ell \ge 3$ and some $m \in \N_{>0}$.
By Lemma \ref{lem:gerbe} and the vanishing of the Brauer group
of a finite field, 
we see that $z$ corresponds to an
absolutely irreducible local $\F_{\ell^m}$-system $\mathbb L_z$ over $X$.

\medskip
We define $S^\circ\subset S$ to be the non-empty open subscheme  which is the
complement of closed points of residual characteristic dividing
the order of $\GL_r(\F_{\ell^m})$. We claim $S^\circ$ satisfies the assertions
of Theorem~\ref{thm:arithm}. Pick $s \in |S^\circ|$ of characteristic $p$
and consider the diagram
\ga{}{sp^{\rm top} _{\C,\bar s} :
\pi_1(X(\C), x(\C))\to \pi_1(X_\C, x_\C)
\xrightarrow{sp_{\C, \bar s}}
\pi_1^t(X_{\bar s}, x_{\bar s}) \notag}
from the introduction. Since $sp_{\C,\bar s}$ is an isomorphism on prime to $p$
quotients, we see that $\mathbb L_z$ gives rise to an
absolutely irreducible local $\F_{\ell^m}$-system $\mathbb L_{z, \bar s}$
over $X_{\bar s}$. Its determinant is $\sL$ and its
monodromies at infinity are in $T_i$ by the compatibility of $sp_{\C, \bar s}$
with the local fundamental groups, see \cite[Section~1.1.10]{Del73}.

\medskip
Let $D_{z, \bar s} = {\rm Spf}(R_{z, \bar s})$
where $R_{z, \bar s}$ is Mazur's formal deformation ring
of the rank $r$ representation $\rho_{z, \bar s}$ of
$\pi_1^t(X_{\bar s}, x_{\bar s})$
over $\F_{\ell^m}$ corresponding to $\mathbb L_{z, \bar s}$
(\cite[Proposition~1]{Maz89}). Consider the formal closed
subscheme $D_{z, \bar s}(r, \mathcal{L}, T_i) \subset D_{z, \bar s}$
corresponding to deformations where the universal deformation has
determinant $\sL$ and monodromies at infinity in $T_i$ (it 
is indeed a closed condition.)  By construction
we obtain a morphism
\ga{}{D_{z, \bar s}(r, \mathcal{L}, T_i) \to
\sM(X, r, \sL, T_i) \to M_B(X, r, \sL, T_i)\notag}
Thus we obtain
\ga{}{D_{z, \bar s}(r, \mathcal{L}, T_i)
\xrightarrow{\iota} M_B(X, r, \sL, T_i)^\wedge_z \notag}
where $(-)^\wedge_z$ indicates the formal completion at $z$.   

\begin{prop}
\label{prop:iota}
The morphism $\iota$ is an isomorphism.
\end{prop}

\begin{proof}
Let $R$ be an Artinian local ring whose residue field is identified
with the residue field $\F_{\ell^m}$ of $z$. To construct the inverse
to $\iota$ we will show that morphisms 
\ga{}{ m : \Spec(R) \to M_B(X, r, \sL, T_i) \notag}
which send the closed point to $z$, are in one to one correspondence 
with deformations of $\rho_{z, \bar s}$ in $D_{z, \bar s}(r, \mathcal{L}, T_i)$.
There are two steps.
First, by Lemma \ref{lem:gerbe} the morphism $m$ lifts to a morphism
$m : \Spec(R) \to \sM(X, r, \sL, T_i)$ into the stack (as the Brauer
group of $R$ is trivial) and moreover the isomorphism class
of the lift is well defined. This lift defines a local $R$-system
$\mathbb L_m$ on $X$ such that
$\mathbb L_R \otimes_R \mathbb F_{\ell^m} = \mathbb L_z$.
Second, by exactly the same arguments as above, this
descends to a local $R$-system $\mathbb L_{m, \bar s}$ on $X_{\bar s}$
(unique up to isomorphism) with determinant $\sL$ and monodromies
in $T_i$ at infinity. The corresponding continuous representation
$\rho_{R, \bar s} : \pi_1^t(X_{\bar s}, x_{\bar s}) \to \GL_r(R)$
is the desired deformation.
\end{proof}

\begin{cor}
\label{cor:sm}
The reduced deformation space $D_{z, \bar s}(r, \mathcal{L}, T_i)_{red}$
is smooth over $ {\rm Spf}(W(\F_{\ell^m}))$.
\end{cor}

\begin{proof}
This follows as by the choices made above the morphisms
\ga{}{ M^\circ \to {\rm Spec}(\sO_K) \ {\rm 
and}  \ {\rm Spec}(\sO_K)^\circ \to {\rm Spec}(\Z)\notag} 
are smooth and the scheme $M^\circ$ is an open subscheme of the
$M_B(X, r, \sL, T_i)_{red}$.
\end{proof}

\medskip
Recall that $D_{z, \bar s}(r, \mathcal{L}, T_i)$ is a deformation space
for the residual representation $\rho_{z, \bar s}$ of
$\pi_1^t(X_{\bar s}, x_{\bar s})$. We have the  homotopy exact sequence
$$
1 \to
\pi_1^t(X_{\bar s}, x_{\bar s}) \to
\pi_1^t(X_s, x_s) \to
\Gal(\bar s/ s) \to 1
$$
of profinite groups. Denote $\Phi \in \Gal(\bar s/ s)$ the Frobenius of $s$.
By the outer action of $ \Gal(\bar s/ s) $ on $\pi_1^t(X_{\bar s}, x_{\bar s})$,
 we see that $\Phi$ acts on the set of isomorphism classes of
representations of $\pi_1^t(X_{\bar s}, x_{\bar s})$.
Since $\rho_{z, \bar s}$ has finite image, a power $\Phi^n$ for some $n \geq 1$
of $\Phi$ stabilizes it. Thus  there is an action of $\Phi^n$ on $D_{z, \bar s}$.
 On the other hand,  $\Phi(\sL)$ is another torsion local system of the same order. Since there are finitely many of them, a power $\Phi^{nm}$ for some $m\ge 1$ stabilizes 
  $\sL$.  By 
\cite[XIV.1.1.10]{SGA7.2} (see also \cite[Lemma 7.1]{EK22}), the action of $\Phi$ on 
the conjugacy class $t_i$ of the  monodromies at infinity is via the cyclotomic character. 
 So $\Phi$ acts on the set of conjugacy classes of quasi-unipotent matrices 
in $GL_r(\bar \Q_\ell)$ with eigenvalues being powers of the eigenvalues of the $T_i$.  Thus for some $t\ge 1$, $\Phi^{nmt}$ stabilizes  $(D_{z, \bar s}, \sL, T_i)$ and thus the formal closed subscheme
$D_{z, \bar s}(r, \sL, T_i) \hookrightarrow D_{z,\bar s}(r)$. We
abuse  notation setting $n = nmt$. So,  $\Phi^n$ acts on
$D_{z, \bar s}(r, \sL, T_i)$.

\medskip
We now apply de Jong's conjecture \cite[Conjecture~2.3]{dJ01}, proved
by Gaitsgory \cite{Gai07} for $\ell\ge 3$, in the way Drinfeld
did in \cite[Lemma~2.8]{Dri01}: Corollary~\ref{cor:sm} implies that
there is a $\bar \Z_\ell$-point of $D_{z, \bar s}(r, \sL, T_i)_{red}$
which is invariant under $\Phi^n$. This corresponds to an irreducible
$\ell$-adic local system $\mathbb L_{z, \bar s, \ell}$ on $X_{\bar s}$,
which descends to a Weil sheaf, thus by  \cite[Proposition~1.3.14]{Del80} to an 
arithmetic \'etale (we just say arithmetic in the sequel) local system
with determinant $\sL$ and monodromies $T_i$ at infinity.
This yields Theorem~\ref{thm:arithm} 1), 2), 3) 5) for this one $\ell$.

\medskip
We now use the method of \cite[Section~3]{EG18}.
For any $\ell' \neq p$ and any algebraic isomorphism
$\sigma : \bar \Q_\ell \cong \bar \Q_{\ell'}$, we denote
$\mathbb L_{z, \bar s, \ell}^\sigma$ the restriction to $X_{\bar s}$
of the companion of the arithmetic descent of $\mathbb L_{z, \bar s, \ell}$.
It is an $\ell'$-adic local system on $X_{\bar s}$
which is irreducible over $\bar \Q_{\ell'}$
(\cite[Proof~of~Theorem~1.1]{EG18}), has determinant $\sL$
by compatibility of companions and exterior powers,
and monodromies at infinity $T_{z, \bar s, \ell', i}$
such that $T_i^{ss} = T_{z, \bar s, \ell', i}^{ss}$
by \cite[Th\'eor\`eme~9.8]{Del72}. We set
$\mathbb L_{\ell, \bar s} = \mathbb L_{z, \bar s, \ell}$ and
$\mathbb L_{\ell', \bar s} = \mathbb L_{z, \bar s, \ell}^\sigma$
for all $\ell' \neq p, \ell$. For $\ell = p$ we just choose another
closed point $s'$ in $S^\circ$ of residual characteristic $p'$ not equal
to $p$, redo the same construction ane set
$\mathbb L_{p, \bar s'}=  \mathbb L_{z, p, \bar s'}$.  This finishes the proof. 

\subsection{Proof of Theorem~\ref{thm:top}}
For any $\ell \neq p$ we set $\mathbb L_\ell=(sp^{\rm top} _{\C,\bar s})^*\mathbb L_{\ell, \bar s}$ and for $\ell=p$ we set
$\mathbb L_{p}=(sp^{\rm top} _{\C,\bar s'})^*\mathbb L_{p, \bar s'}$. This finishes the proof.

\section{Density of weakly  arithmetic  local systems} \label{sc:dens}

In this section we use the notation of Section~\ref{sec:top}.  Our aim is to define a notion of weakly arithmetic local systems and to show density of them in the Betti moduli space.

\subsection{Definitions}

If $\tau : K_1 \to K_2$ is a homomorphism of fields, and $\mathbb L_1$
is a topological local system defined over $K_1$   by the homomorphism $ \rho: \pi_1(X(\C), x(\C))\to GL_r(K_1)$, we denote by
$\mathbb L_{K_1}^\tau$ the local system defined over $K_2$
obtained by post-composing  $ \rho$ by the homomorphism $ GL_r(K_1) \to GL_r(K_2)$ defined by  $\tau$. 

\medskip
Choose a finitely generated subfield $F \subset \C$ such that
$X$ and $x$ descend to $X_F$ and $x_F$ over $F$. For example $F$
could be the function field of the affine finite type scheme $S$
over which $X$ has a model.
Denote 
\ga{}{sp^{\rm top} _{\C, F} : \pi_1(X(\C), x(\C))\to \pi_1(X_F, x_F) \notag}
the comparison morphism. Recall that $M_B(X,r,\sL, T_i)$ is defined over the number ring $\sO_K \subset \C$.
  
\begin{defn}
\label{defn:arithm}
With notation as above:
\begin{itemize}
\item[1)] A point $\mathbb L_{\C} \in M_B(X, r)(\C)$ is said to be arithmetic
if there exist a finite extension $F'/F$, an \'etale $\ell$-adic local system
$\mathbb L_{\ell, F'}$ on $X_{F'}$, and a field isomorphism\footnote{It would
be equivalent to say "homomorphism of fields" here.}
$\tau : \bar \Q_\ell \to \C$ such that   $\mathbb L_\C$   and
$\big((sp^{\rm top}_{\C, F'})^{-1}(\mathbb L_{\ell, F'})\big)^\tau$ are
isomorphic.
\item[2)] A point $\mathbb L_{\C} \in M_B(X,r,\sL, T_i)(\C)$ is said to be
{\it weakly arithmetic} if there exist
\begin{itemize}
\item[i)] a prime number $\ell$ and a field isomorphism $\tau : \bar\Q_\ell \to \C$  (so  $\tau$  
defines  an $\sO_K$-algebra structure on  $\bar\Q_\ell$
via $\sO_K \subset \C$);
\item[ii)] a finite type scheme $S$ over $\Spec(\sO_K)$
such that $(X \hookrightarrow \bar X, x,  \sL, T_i)$
has a model over $S$ (see introduction);
\item[iii)] a closed point $s \in |S|$ of residual characteristic
different from $\ell$;
\item[iv)] a tame arithmetic $\ell$-adic local system
$\mathbb L_{\ell, \bar s}$ on $X_{\bar s}$ with determinant
$\sL$ and monodromy at infinity in $T_i$;
\end{itemize}
such that $\mathbb L_\C$ and
$\big((sp^{\rm top}_{\C, \bar s})^{ -1}(\mathbb L_{\ell, \bar s})\big)^\tau$
are isomorphic.
\end{itemize}
\end{defn}
 We can express the definition in an imprecise way by saying the following. 
As  $\pi_1(X(\C), x(\C))$ is finitely generated, the topological local system $\mathbb L_\C$ is defined over a ring $A$ of finite  type over $\Z$, 
so $\mathbb L_\C=\mathbb L_A\otimes_A \C$.
For almost all prime numbers $\ell$, there is a non-zero homomorphism $A\to \bar \Z_\ell$. For any such,  $ \mathbb L_A\otimes_A \bar \Z_\ell$
defines an $\ell$-adic local system. Then $\mathbb L_\C$ is weakly arithmetic if there is such a non-zero $A\to \bar \Z_\ell$ such that $ \mathbb L_A\otimes_A \bar \Z_\ell$
comes via the specialization homomorphism from a tame arithmetic $\ell$-adic local system on $X_s$ for some $s$ of large characteristic.

\begin{rmk}
The notion of an arithmetic local system does not depend on the choice of $F$. 
As we could not find a reference, we give a short argument. 
Let $F_1, F_2 \subset \C$ be two finitely generated subfields over
which $X$ and $x$ can be defined. Then the compositum is a third one.
So we many assume that $F_1 \subset F_2 \subset \C$.
We have the commutative diagram of homotopy exact sequences
$$
\xymatrix{
1 \ar[r] &
\pi_1(X_{\bar F_2}, x_{ \bar F_2}) \ar[r] \ar[d]^{\cong} &
\pi_1(X_{F_2}, x_{\bar F_2}) \ar[r] \ar[d] &
{\rm Gal}(\bar F_2/F_2) \ar[r] \ar[d] & 1 \\
1 \ar[r] &
\pi_1(X_{\bar F_1}, x_{\bar F_1}) \ar[r] &
\pi_1(X_{F_1}, x_{\bar F_1}) \ar[r] &
{\rm Gal}(\bar F_1/F_1) \ar[r] & 1
}
$$
Since $F_2/F_1$ is a finitely generated extension, the right vertical map
is open. Let  $\mathbb L$ be a geometrically irreducible arithmetic
$\ell$-adic local system over   $X_{\bar F_2}$ with underlying representation $\rho$.   Thinking of $X$ as being defined over $F_2$, 
by replacing $F_2$ by a finite extension, we may assume that   $\rho$ is defined over $F_2$. Then  $K:={\rm Ker} ({\rm Gal}(\bar F_2/F_2)\to {\rm Gal}(\bar F_1/F_1))$ lifts to $\pi_1(X_{F_2}, x_{\bar F_2}) $ and stabilizes $\mathbb L|_{X_{\bar F_2}}$, which is irreducible. Thus $\rho(K) \subset \bar \Q_\ell^\times\subset \GL_r(\bar \Q_\ell)$ and is torsion as the determinant is torsion. This defines a finite (abelian) extension of $F_2$.  Replacing $F_2$ by it, $\rho$ factors through $\pi_1(X_{F_2}, x_{\bar F_2})/K$ which is equal to $\pi_1(X_{F'_1}, x_{\bar F_1})$ where $F_1\subset F'_1$ is the finite Galois extension defined by the image of the right vertical map which is open.
\end{rmk}

\begin{rmk}
An arithmetic local system is weakly arithmetic.
Namely, suppose that $\mathbb L_\C$ is arithmetic.
Let $(S, X_S \hookrightarrow \bar X_S, x_S)$ be as in the introduction.
Denote by  $F$ the function field of $S$. Then $\mathbb L_\C$
comes from an \'etale $\ell$-adic local system
$\mathbb L_{\ell, F'}$ on $X_{F'}$ for $F'/F$ finite.
After replacing $S$ by a finite cover, we may assume $F' = F$.
A standard argument shows that after replacing $S$ by a suitable
Zariski open, we may assume that $\mathbb L_{\ell, F}$
comes from an \'etale $\ell$-adic local system
$\mathbb L_{\ell, S}$ on $X_S$
(see e.g. \cite[Proposition~6.1]{Pet20} where the argument
is performed for $S$  the spectrum of a number ring, and references therein).
Then, for a closed point $s \in S$
we see that $\mathbb L_\C$ comes from
$\mathbb L_{\ell, \bar s} := \mathbb L_{\ell, S}|_{X_{\bar s}}$
as desired. We omit the details.
\end{rmk}

\begin{nota}
1) Fixing $(X,r, \sL, T_i)$ as in Section~\ref{sec:top}, we denote by 
 \ga{}{ W(X,r,\sL, T_i) \subset M_B(X,r,\sL, T_i)(\C) \notag}
 the locus 
 of weakly arithmetic local systems.\\ 
 2)   Fixing $(X,r, \sL)$, we denote by 
 \ga{}{ W(X,r,\sL) = \cup_{\{ T_i\}} W(X,r,\sL, T_i)  \subset M_B(X,r, \sL)(\C) \notag}
  the locus of all  weakly arithmetic local systems of rank $r$ and determinant $\sL$.
\end{nota}
\subsection{Density}

\begin{thm}
\label{thm:dens} 
1) Fixing $(X,r,\sL, T_i)$,  $W(X,r,\sL, T_i)$ is dense in $M_B(X,r,\sL, T_i)(\C)$. \\
2) Fixing $(X,r, \sL)$, $W(X,r,\sL)$ is dense in $M_B(X,r,\sL)(\C) $.
\end{thm}

\begin{proof}
Ad 1):  Let $T_\C$ be  the Zariski closure 
  of $W(X,r,\sL, T_i)$ in $M_B(X,r,\sL, T_i)(\C)$ and $T_{\sO_K}$ be the Zariski closure of $T_\C$ in $M_B(X,r,\sL, T_i)$.  As $T_\C$ is invariant under the group ${\rm Aut}_{\sO_K}(\C)$ of field automorphisms of $\C$ over $\sO_K$,  $T_{\sO_K}$ is the Zariski closure of $W(X,r,\sL, T_i)$  in $M_B(X,r,\sL, T_i)$.

If $T_\C \neq M_B(X,r,\sL, T_i) (\C)$, 
one chooses a
closed point $z \in M_B(X, r, \mathcal{L}, T_i)$ in 
 Section~\ref{pf:non-resp} outside of $T_{\sO_K}$. Then $(sp_{\C,\bar s}^{\rm top})^{-1}( \mathbb L_{\ell, \bar s, z})$ constructed in
 Section~\ref{pf:non-resp}, with $\mathbb L_{\ell, \bar s, z}$ arithmetic on $X_{\bar s}$,  does not lie in $T_{\sO_K} (\bar \Q_\ell)$. By invariance, for  any $\tau$, 
 $((sp_{\C,\bar s}^{\rm top})^{-1}( \mathbb L_{\ell, \bar s, z}))^\tau$ does not lie on $T(\C)$, a contradiction to the definition of weak arithmeticity. 

\medskip
\noindent
Ad 2): By \cite[Theorem~1.3]{EK23}, for a torsion rank $1$ local system $\sL$ given,  
 \ga{}{ \cup_{T_i}  M_B(X, r, \sL, T_i) \notag}  is dense in 
$M_B(X,r,\sL) (\C)$. Combined with 1),  this yields 2).

\end{proof}

\begin{rmk}
In \cite[Weak~Conjecture]{EK22} it is predicted that arithmetic
$\ell$-adic local systems on a smooth quasi-projective variety
$X_\F$ defined over a finite field $\F$ of characteristic different
from $\ell$ are Zariski dense in a  Mazur, or more generally, in a Chenevier
deformation space. In \cite[Conjecture~1.1]{EK23} it is predicted that
arithmetic  local systems in $M_B(X,r, \sL)(\C)$ are Zariski dense,
where $\sL$ is a torsion rank one local system.
{\footnote{In fact in \cite{EK23} is expressed for local systems of geometric origin.}}
This is not correct, see
\cite[Theorem~8.1.2]{LL22} in which Landesman-Litt prove that
over a very general genus $\ge 3$ curve, for $r$ small, arithmetic
local systems have finite monodromy,  and \cite[Corollary~1.2.10]{LL22a}  for the consequence on non-density.
Theorem~\ref{thm:dens} yields a replacement for this. 
However, there  may be  uncountably many  weakly  arithmetic local systems due to the free choice of $\tau$.
\end{rmk}

\section{Proof of Theorem~\ref{thm:obs}} \label{sec:obst}

\subsection{Proof}[Proof of Theorem~\ref{thm:obs}.] \label{sec:obs}
 We have to prove that $\Gamma = \pi_1(X(\C), x(\C))$ has the weak 
integrality property (Definition~\ref{defn:weak})  with respect to any $(r,\chi)$ where $r\ge 1$ is a natural number and 
$\chi$ is a character of $\Gamma$. 

Let $\rho : \Gamma \to \GL_r(\C)$ be
as in Definition \ref{defn:weak}. Then
the density of weakly arithmetic local systems in
Theorem \ref{thm:dens} shows that we may assume $\rho$
corresponds to a weakly arithmetic local system $\mathbb L_\C$.
Since such a system has quasi-unipotent local monodromies
at infinity by Grothendieck's theorem ~\cite[Appendix]{ST68}, we conclude that $\mathbb L_\C$
is as in Theorem~\ref{thm:top} for some choice of quasi-unipotent
conjugacy classes $T_i$. Thus by  Theorem~\ref{thm:top}, we find a representation
$\rho_\ell : \Gamma \to \GL_r(\bar\Z_\ell)$ with determinant
$\chi$ which is irreducible over $\bar\Q_\ell$.

\medskip
Next, the discussion in the introduction shows that the group $\Gamma_0$  presented after Definition~\ref{defn:weak}
does not have the  weak integrality property for $r=2$ and $\chi$ being the trivial character. 
 Thus the weak integrality property for $\Gamma$ is a non-trivial obstruction for it to be of the shape $\pi_1(X(\C), x(\C))$. 
 The proof of
Theorem~\ref{thm:obs} is finished.

\subsection{Comments}
1) In \cite[Theorem~1.1]{Kli19}, the main integrality theorem \cite[Theorem~1.1]{EG20} for cohomologically rigid local systems is used as  {\it the} criterion   to decide that certain $p$-adic lattices are not the fundamental group of a smooth projective variety (or are not K\"ahler groups).  Recall that Theorem~\ref{thm:top} of the present article relies in part on the proof of \cite[Theorem~1.1]{EG20}, but for the other part on de Jong's conjecture. The obstruction we obtain in Theorem~\ref{thm:obs} is of a new kind, it enables one consider all quasi-projective varieties at once, with all boundary conditions and not only smooth projective varieties. \\ \ \\
2)  In \cite[Section~4]{dJEG22} examples of local systems $\mathbb L_\C$ are constructed with the following property: they lie in $M_B(X,r,\sL, T_i)(\C)$ for some $\sL$ and $T_i$, are 
in this Betti moduli  with boundary conditions  cohomologically rigid, and viewed in $M_B(X,r, \sL)(\C)$, they are still rigid but no longer cohomologically rigid. 
 So in a way this is good to keep track of the conditions at infinity.

\section{Proof of Theorem~\ref{thm:arithm}, resp.\ case} \label{sec:resp}
\label{pf:resp}  The goal of this section is to prove the resp. case of 
Theorem~\ref{thm:arithm}. That is we have to show that we can construct   infinitely many  $\ell$-adic local systems 
as  Section~\ref{sec:top}  if we start with infinitely many $\mathbb L_\C$. 
The existence of infinitely many $\mathbb L_\C$ implies
there is an irreducible component of $M_B(X, r, \sL, T_i)$
 the base change of which  to $K$ has dimension $> 0$. Thus, after replacing
$K$ by a finite extension  over which all components are defined, we may assume that there is an irreducible
component $Z \subset M_B(X, r, \sL, T_i)$ such that $Z_K$ is
geometrically irreducible and of dimension $> 0$. Then $Z$ is an
integral scheme and the morphism $Z \to \Spec(\sO_K)$ is dominant and of
relative dimension $\ge 1$. Let $Z^\circ \subset Z$ be a nonempty open
subscheme which is smooth over $\Spec(\sO_K)$ which maps into the
open $\Spec(\sO_K)^\circ$ of absolutely unramified points. We also may
and do assume that $Z^\circ$ does not meet any irreducible
component of $M_B(X, r, \sL, T_i)$ except $Z$. Finally, we may
further shrink $\Spec(\sO_K)^\circ$ such that all fibres of
$Z^\circ \to \Spec(\sO_K)^\circ$ are geometrically irreducible.

\medskip
We redo the argument of the proof~\ref{pf:non-resp}
with $Z^\circ$ replacing $M^\circ$. Namely, we pick $z \in |Z^\circ|$ closed
with residue field $\F_{\ell^m}$. This determines an $\F_{\ell^m}$
local system $\mathbb L_z$ on $X$.
Denote $t \in \Spec(\sO_K)$ be the image of $z$ and recall that $Z^\circ_t$
is smooth and geometrically irreducible of dimension $\ge 1$.
We let $S^\circ \subset S$ be the open where the residue characteristics
are prime to $|\GL_r(\F_{\ell^m})|$. Next, let $s \in |S^\circ|$ of
residue characteristic $p$. Since,  for any finite field $\F$,  $|\GL_r(\F)|$ is a polynomial
in $|\F|$,  it follows that $|\GL_r(\F_{\ell^{mp^N}})|$ is prime to $p$
for all $N \geq 0$. By the   Lang-Weil estimates
$$
\sZ = \bigcup\nolimits_{N \ge 0} Z^\circ_t(\F_{\ell^{mp^N}})
$$
is infinite. Thus we may choose an infinite sequence $z_\alpha \in Z^\circ_t$
of closed points such that $|\GL_r(\kappa(z_\alpha))|$ is prime to $p$.
Here $\kappa(z_\alpha)$ denotes the residue field.
Hence, all the local systems $\mathbb L_{z_\alpha}$
give rise to local systems $\mathbb L_{z_\alpha, \bar s}$ on $X_{\bar s}$.
Redoing the construction of  Section~\ref{pf:non-resp}
we find irreducible $\ell$-adic systems
$\mathbb L_{z_\alpha, \bar s, \ell}$ on $X_{\bar s}$,
which have arithmetic descent, with determinant $\sL$ and monodromies
$T_i$ at infinity. These systems are pairwise non-isomorphic, as their
mod $\ell$ reductions $\mathbb L_{z_\alpha, \bar s}$ are pairwise
non-isomorphic and  absolutely irreducible. For a prime $\ell' \not = p$ we fix
an isomorphism $\sigma : \bar \Q_\ell \cong \bar \Q_{\ell'}$.
Then the companions $\mathbb L_{z_\alpha, \bar s, \ell}^\sigma$ constructed
in   Section~\ref{pf:non-resp} are likewise pairwise non-isomorphic. Indeed, 
assume $\mathbb L_{z_\alpha, \bar s, \ell}^\sigma$ and $\mathbb L_{z_\beta, \bar  s, \ell}^\sigma$
are isomorphic. Let us choose a point $s'\to s$ such that the residue field extension $\kappa(s)\hookrightarrow \kappa(s') (\hookrightarrow \kappa(\bar s))$ is finite,
 so the  arithmetic descents  $\mathbb L_{z_\alpha,  s', \ell}^\sigma$ 
and  $\mathbb L_{z_\beta,  s', \ell}^\sigma$ are defined. So they differ by a character of   $\kappa(s')$.  Thus $\mathbb L_{z_\alpha, \bar s, \ell}$ and $\mathbb L_{z_\beta, \bar s, \ell}$ descend to $X_{s'}$  over which they  differ by a character of $\kappa(s')$. So they are isomorphic.  Furthermore, they have the same determinant $\sL$ and the same semi-simplification of the monodromies at infinity.

\medskip
In the same manner we may deal with the case $\ell' = p$; we omit
the details. This finishes the proof.

\section{Proof of Theorem~\ref{thm:type}   i) ii)} \label{sec:dens}
 The aim of this section is to prove  Theorem~\ref{thm:type}  i) ii).  The first part i), which comes from one part of the proof of the integrality for cohomologically rigid local systems,   is entirely contained in \cite{EG18}  and is repeated here for the reader's convenience. We also insert a comment  on Theorem~\ref{thm:type} i), see Remark~\ref{rmk:pigeon}.
 The second part ii) is new. It says that the monodromy of the 
$\ell$-adic local systems constructed in Theorem~\ref{thm:top} is large if the initial topological local system $\mathbb L_\C$ has large monodromy.

 \subsection{Part i)}
 
 If in Theorem~\ref{thm:arithm}, we assume that $\mathbb L_{\C}$
is cohomologically rigid, then  the topological local system
$(sp^{\rm top} _{\C,\bar s})^*(\mathbb L_{\ell, \bar s})$
of the statement of Theorem~\ref{thm:type} i) is constructed in this article 
exactly as in \cite{EG18}.  We give here a brief account:   $\mathbb  L_\C$ is defined over a ring of finite type $A$, 
so $\mathbb L_C= \mathbb L_A\otimes_A \C$,  and we first take a  non-zero homomorphism $A\to \bar \Z_{\ell_0}$, where $\ell_0$ is a prime number. This defines an $\ell_0$-adic local system
$\mathbb L_{\ell_0}$ which descends, as all $\ell$-adic local systems do, to $X_{\bar s}$ for $p={\rm char}(s)$ large. Call it $\mathbb L_{\ell,_0 \bar s}$. 
By the very definition, $(sp^{\rm top} _{\C,\bar s})^*(\mathbb L_{\ell_0, \bar s})$ has an $A$-model which via the embedding $A\subset \C$ gives back $\mathbb L_\C$.
Cohomological  rigidity implies that  $\mathbb L_{\ell_0, \bar s}$ descends to $\mathbb L_{\ell_0,  s}$ so we do not need de Jong's conjecture to have this arithmetic descent. 
Choose a field isomorphism  $\sigma: \bar \Q_{\ell_0} \xrightarrow{\cong }\bar \Q_\ell$ for some prime $\ell$ prime to $p$. Cohomological  rigidity  also implies that 
the companions $\mathbb L_{\ell_0, s}^\sigma= :  \mathbb L_{\ell, s}$  has the property that 
$(sp^{\rm top} _{\C,\bar s})^*(\mathbb L_{\ell, \bar s})$
is cohomologically rigid.  So  for $\sigma$ given and  each such $\mathbb L_\C$, we produce a topological local system $(sp^{\rm top} _{\C,\bar s})^*(\mathbb L_{\ell, \bar s})$
which comes from an \'etale local system, is cohomologically rigid (with all the extra conditions preserved, determinant, monodromies at infinity). This finishes the proof.

\begin{rmk} \label{rmk:pigeon}
The statement and the proof  of Theorem~\ref{thm:type} i)  do  not explain the relation between the initial $\mathbb L_\C$ and $(sp^{\rm top} _{\C,\bar s})^*(\mathbb L_{\ell, \bar s})$.
To do this,  in  \cite{EG18}  we argued that
  two non-isomorphic such 
$\mathbb L_\C$ yield two non-isomorphic $(sp^{\rm top} _{\C,\bar s})^*(\mathbb L_{\ell, \bar s})$. So by finiteness of the set of such $\mathbb L_\C$, they are integral. As  rigidity (alone) implies that $A$ could have been taken from the beginning to be the localization at finitely many places of a number ring, we conclude that $A$ can be taken to be a number ring. But we can not conclude that $\mathbb L_A\otimes_A\bar \Q_\ell$ is the system $(sp^{\rm top} _{\C,\bar s})^*(\mathbb L_{\ell, \bar s})$ constructed earlier. 
\end{rmk}

 \subsection{Part ii)} 
Consider the set
$$
Z_K \subset M_B(X, r, \sL, T_i) \times_{\Spec(\sO_K)} \Spec(K)
$$
of points parametrizing local systems such that the Zariski
closure of the monodromy does not contain $\SL_r$. This set is closed,
because the complement $W_K$ is open by \cite[Theorem 8.2]{AB94}.
Let $Z \subset M_B(X, r, \sL, T_i)$ be the Zariski closure of $Z_K$.
Then
$$
W = M_B(X, r, \sL, T_i) \setminus Z
$$
is an open subscheme with $W_K = W \times_{\Spec(\sO_K)} \Spec(K)$, i.e.,
all of the points of $W$ in characteristic zero
correspond to local systems such that the Zariski closure of
the monodromy contains $\SL_r$. So it is non-empty by assumption.
Then in the arguments of \ref{pf:non-resp} if we pick $z \in W$
we  see that our arithmetic $\ell$-adic local system
$\mathbb L_{z, \bar s, \ell}$ pulls back to a topological local
system on $X$ which corresponds to a point of $W_K$ and hence
has Zariski closure of the monodromy containing $\SL_r$.

\medskip
\noindent
Having constructed one arithmetic $\ell$-adic local system of rank $r$ on $X_{\bar s}$
with large monodromy (in the sense that the Zariski closure of the
image of monodromy contains $\SL_r$) with determinant $\sL$ and monodromies
in $T_i$ we can take the companions and show they have large monodromy too.
For this one can use known facts on compatible systems of Galois
representations, see for example \cite[Theorem~1.2.1]{D'Ad20} (where we disregard the crystalline statement)  which guarantees that
the companion of an $\ell$-adic local system with large monodromy also
has large monodromy.
\begin{rmk}

We can
argue as above and as in Section \ref{pf:resp} to prove a variant
of Theorem~\ref{thm:type} ii) for infinite collections.
We omit the detailed formulation and proof.
\end{rmk}

\section{Remarks and Questions} \label{sec:RQ}
\subsection{The crystalline version} 
It is  expected   that there are $\ell$ to $p$ companions in the sense of Deligne, see \cite[Conjecture~1.2.10]{Del80} for the original formulation and \cite[Definition~1.4 5)]{AE19} for a precise formulation. If so, Theorem~\ref{thm:arithm} implies what is as of today a conjecture{\color{red}.}
\begin{conj}
If there is one (resp.\ infinitely many pairwise non-isomorphic)  irreducible topological rank $r$ complex local system  (resp.\ systems) $\mathbb L_\C$ with  torsion determinant  $\sL$ and quasi-unipotent monodromies  in $T_i$ at infinity, 
 then 
there is a non-empty open subscheme $S^\circ\subset S$ such that for any  closed point $s \in |S|$ 
 it holds:
{\rm  there is one (resp.\ infinitely many pairwise non-isomorphic ) Frobenius invariant iscrystal $M_{s'}$ on $X_{s'}$ (resp.\ infinitely many pairwise non-isomorphic Frobenius invariant isocrystals $M_{s_\alpha}$, each defined on some $X_{s_\alpha}$), where $s'\to s$ and $s_\alpha\to s$ are closed points, with determinant $\sL$ (as an $F$-isocrystal) and residues modulo $\Z$ being along $D_{i, \bar s}$ the log of the eigenvalues of $T_i$. }

\end{conj}

\subsection{Specific subloci}
\label{spe_subl}

Recall that Simpson's integrality conjecture~\cite[p.9]{Sim92}
is proven only for cohomologically rigid local systems \cite[Theorem~1.1]{EG18}
while there are rigid non-cohomologically rigid local systems~ \cite{dJEG22}.

\begin{question}
Given a locally closed subset $W \subset M_B(X, r, \sL, T^{ss}_i)$
does Theorem~\ref{thm:arithm} hold with the following modifications
\begin{itemize}
\item[a)] $\mathbb L_\C$ is assumed to lie in $W(\C)$;
\item[b)] the resulting $(sp^{\rm top} _{\C,\bar s})^*(\mathbb
L_{\ell, \bar s} )\otimes_{\bar \Z_\ell} \bar \Q_\ell $ should  lie in
 $W (\bar \Q_\ell)$?
\end{itemize}
\end{question}

 The answer will not be positive for all $W$; we should only consider
suitably natural loci in the moduli space. For example,
Theorem~\ref{thm:type} i) says that the answer is   ``yes"  if $W$ is
the union of the
smooth  isolated points of $M_{dR}(X,r, \sL, T^{ss} _i)(\C)$, i.e. the
underlying $\mathbb L_\C$ are cohomologically rigid.
Theorem~\ref{thm:type} ii) says the answer is    ``yes" if
we take $W$ to be the locus of points in
$M_{dR}(X, r, \sL, T^{ss}_i)(\C)$ such that the Zariski closure of
monodromy contains $\SL_r$.

\medskip
On the other hand, Simpson conjecture says the answer is   ``yes"  if
$W$ is the union of the isolated points of
$M_{dR}(X,r, \sL, T^{ss} _i)(\C)$, i.e., the
underlying $\mathbb L_\C$ are rigid. As a tentative generalization of
Simpson's conjecture, we ask if the answer is "yes" when $W$ is
the union of the  irreducible components of $M_B(X,r,\sL, T^{ss}_i)(\C)$
of a given dimension.

\subsection{Pullbacks of semisimple local systems}
Using our methods we can reprove the following theorem. Its 
  only
known proof, as far as we understand, uses the existence of a 
tame pure imaginary harmonic metric. 

\begin{thm}\cite[Theorem~25.30]{Moc07} \label{thm:Moc} \footnote{There is a typo in
{\it loc.cit.} \ where the condition of normality is dropped. The reduction
to the smooth case is explained in the proof of Theorem~\ref{thm:Moc} .}
Let $f : Y \to X$ be a morphism of normal quasi-projective complex varieties.
The pullback by $f$ of a semi-simple complex local system $\mathbb L_\C$ is
semi-simple.
\end{thm}

We sketch a proof of the theorem using our techniques.
The subtle part is that ``being semi-simple'' is neither a closed nor an open
condition in moduli of representations, but it is constructible
and this is enough for our arguments. We briefly indicate the proof strategy.

\begin{proof}
Let us reduce to the case where $X$ and $Y$ are smooth. Let $g : X' \to X$
be a resolution of singularities. Note that $g^*\mathbb L_\C$ is semi-simple
as $\pi_1(X'(\C)) \to \pi_1(X(\C))$ is surjective due to the fact that $X$
is normal. Since $X'\times_X Y \xrightarrow{{\rm projection}} Y$ is surjective,
we can find a generically finite morphism of varieties $h : Y' \to Y$
with $Y'$ smooth and
a morphism $f' : Y' \to X'$ such that $g \circ f' = f \circ h$.
 Here is the corresponding picture
$$
\xymatrix{
Y' \ar[r]_{f'} \ar[d]_h &
X' \ar[d]^g \\
Y \ar[r]^f & X
}
$$
After replacing
$Y'$ by a generically finite cover, we may asume the function field extension
$\C(Y')/\C(Y)$ is Galois. Then there exists an open $Y^\circ \subset Y$
which is smooth and such that
$(Y')^\circ = h^{-1}(Y^\circ) \to Y^\circ$ is finite \'etale
and Galois. Since $Y$ is normal the map
$\pi_1(Y^\circ(\C)) \to \pi_1(Y(\C))$ is surjective
and hence it suffices to show that $f^*\mathbb L_\C|_{Y^\circ}$
is semi-simple. Since $(Y')^\circ \to Y^\circ$ is Galois it
is enough to show that $h^*f^*\mathbb L_\C|_{(Y')^\circ}$ is
semi-simple. Since $h^*f^*\mathbb L_\C = (f')^*g^*\mathbb L_\C$
we see that it suffices to prove the result for $f'|_{(Y')^\circ}$.

\medskip
From now on we assume $X$ and $Y$ smooth.
To obtain a contradiction, we assume that $\mathbb L_\C$ is irreducible
and that we are given a subsystem $\mathbb L' \subset f^*\mathbb L_\C$ of
rank $0 < r' < r$ which does not split off.
Consider the morphisms of Betti moduli spaces
\ga{}{ (\star) \ \ 
M_B^{split}(f , r', r)
\longrightarrow
M_B(f , r', r)
\notag
}
Here the moduli space $M_B(f, r', r)$ parametrizes pairs
$(\mathbb L, \mathbb L' \subset f^*\mathbb L)$ consisting of
an irreducible local system $\mathbb L$ over $X$  together
with a rank $r'$ subsystem $\mathbb L'$ over $Y$.
And the moduli space $M_B^{split}(f , r', r)$
parametrizes triples $(\mathbb L, \mathbb L' \subset f^*\mathbb L, \tau)$
where $\tau : f^*\mathbb L \to \mathbb L'$ is a splitting of the
inclusion map.

\medskip
The moduli schemes in $(\star)$ are finite type schemes over $\Z$; this
follows from arguments similar to those in Subsection \ref{sec:Betti}.
Hence, by Chevalley's theorem the image of $(\star)$
is a constructible set. Let $W \subset M_B(f , r', r)$
be the complement of the image; this is also a constructible subset.
The assumption that we have $\mathbb{L}_\C$
and $\mathbb{L}' \subset f^*\mathbb{L}_\C$ tells us that $W$
has a characteristic zero point. Let $W' \subset W$ be a subset
which is an irreducible, locally closed subset of $M_B(f, r', r)$
containing a generic point of $W$ of characteristic zero.
After replacing $W'$ by an open subset,
we may and do assume that $W'$ does not meet the closure of $W \setminus W'$.
We view $W'$ as a reduced, irreducible, locally closed subscheme
of $M_B(f, r', r)$.   As in Subsection \ref{pf:non-resp} we choose a
closed point $z \in W'$ in the smooth locus of $W' \to \Spec(\mathbb Z)$.
Say $\kappa(z) = \F_{\ell^m}$ for a prime number $\ell \ge 3$
and some $m \in \N_{>0}$.

\medskip
Next, as in the introduction, we choose a model for $f$, i.e., we choose
an integral affine scheme $S$
of finite type over $S$, a morphism $Y_S \to X_S$ of smooth schemes over $S$
such that $Y_S$ and $X_S$ have a good compactifications over $S$, and
such that there is a dominant morphism $\Spec(\C) \to S$ such that
the base change of $Y_S \to X_S$ by this morphism is isomorphic to $f$.
Consider an open $S^\circ$ and a closed point $s \in S^\circ$ as in
Subsection \ref{pf:non-resp}. We obtain an
absolutely irreducible local $\F_{\ell^m}$-system $\mathbb L_{z, \bar s}$
over $X_{\bar s}$ which now comes endowed with a subsystem
$\mathbb L'_{z, \bar s} \subset f_{\bar s}^*\mathbb L_{z, \bar s}$
of rank $r'$ over $Y_{\bar s}$. Next, we consider the deformation space
$$
D_{z, \bar s}(f, r', r)
$$
classifying deformations of $\mathbb L_{z, \bar s}$ endowed with a rank $r'$
subsystem of the pullback to $Y_{\bar s}$ deforming $\mathbb L'_{z, \bar s}$.
The analogue of Proposition \ref{prop:iota} holds in this situation:
the morphism
$$
D_{z, \bar s}(f, r', r) \xrightarrow{\iota} M_B(f, r', r)^\wedge_z
$$
to the formal completion of the moduli space is an isomorphism.
(But this time we do not know or claim
formal smoothness for this deformation space or its reduced structure.)
Since the local system $\mathbb L_{z, \bar s}$ and the subsystem
$\mathbb L'_{z, \bar s} \subset f_{\bar s}^*\mathbb L_{z, \bar s}$
can be defined over a finite extension of $\kappa(s)$, we
obtain an action of $\Phi^n$ on $D_{z, \bar s}(f, r', r)$;
please compare with the discussion following Corollary \ref{cor:sm}.

\medskip
To get a fixed point for the action of $\Phi^n$, we restrict
to deformations lying in $W'$ as follows. Denote
$$
W'_{z, \bar s} \subset D_{z, \bar s}(f, r', r)
$$
the inverse image of the closed formal subscheme
$(W')^\wedge_z \subset M_B(f, r', r)_z^\wedge$ by the isomorphism $\iota$.
Now we claim that $W'_{z, \bar s}$ is invariant under the action
of $\Phi^n$. Namely, $W'_{z, \bar s} \cong (W')^\wedge_z$
is formally smooth over $W(\F_{\ell^m})$ by our choice of $z$.
Hence in order to see that it is stabilized by $\Phi^n$
it suffices to see that its set of $\bar \Z_\ell$-points is stabilized.
However, $\bar \Z_\ell$-points of $(W')^\wedge_z$
are those $\bar \Z_\ell$-points of $D_{z, \bar s}(f, r', r)$
such that the corresponding $\bar \Q_\ell$ pair
$(\mathbb L, \mathbb L' \subset f^*\mathbb L)$
does not split\footnote{Here we need to use the careful
choice of the component $W'$ and the fact that no other
component of the closure of $W$ passes through $z$.}.
Pulling back by the Frobenius automorphism does not
change this property. So $W'_{z, \bar s}$ is invariant under the action
of $\Phi^n$.

\medskip
We conclude as before that we obtain a $\bar \Z_\ell$-point of
$W'_{z, \bar s}$ fixed by $\Phi^n$. This point corresponds to a pair
$(\mathbb M, \mathbb M' \subset f_{\bar s}^*\mathbb M)$ consisting of a
$\bar\Z_\ell$-local system on $X_{\bar s}$ and a subsystem of rank
$r'$ over $Y_{\bar s}$. Since our fixed point lies in $W'$, by definition of
$W'$, the local system $\mathbb{M}$ is irreducible and the inclusion
$\mathbb M' \subset f_{\bar s}^*\mathbb M$ does not split!
Being fixed by $\Phi^n$ exactly signifies that
$\mathbb M$ descends to a Weil sheaf (\cite[Definition 1.1.10]{Del80})
$\mathbb{M}^\circ$ on $X_\kappa$ for a finite extension
$\kappa / \kappa(s)$ and that the inclusion
$\mathbb M' \subset f_{\bar s}^*\mathbb M$ comes from an inclusion
$(\mathbb M')^\circ \subset f_\kappa^*\mathbb M^\circ)$ of Weil sheaves.
We will see that this leads to a contradiction.

\medskip
Namely, by \cite[Proposition~1.3.14]{Del80}, we may assume that
$\mathbb M^\circ$ is arithmetic (``par torsion'').
While there is no condition on the determinant and the monodromies
at infinity in the definition of moduli in $(\star)$, we now conclude
the determinant of $\mathbb M$ is torsion (and the monodromies at infinity
are quasi-unipotent). This means we can further assume $\mathbb M^\circ$
has finite order determinant (again ``par torsion'').
By \cite[Th\'eor\`eme~VII.6]{Laf02} in dimension $1$, and
\cite[Theorem4.4]{EK12}, or \cite[Section 0.7]{Del12}
in higher dimension, $\mathbb M$ is pure of weight $0$, and so is
its pullback to $Y_{\bar s}$. Thus \cite[Lemme~3.4.3]{Del80} implies 
that $f^*\mathbb M|_{Y_{\bar \kappa}}$ is semi-simple.
This is the desired contradiction.
\end{proof}

\subsection{The example of Becker-Breuillard-Varj\'u} \label{BBV} 
In \cite{BBV22}, the authors study the  $\SL(2, \C)$-character variety ${\rm Ch}(\Gamma_0, 2, \mathbb I)$  of irreducible representations of the residually finite  group $\Gamma_0$ defined in \cite[Theorem~4]{DS05}, with two generators  $\{a,b\}$ and one relation $b^2=a^2ba^{-2}.$  They prove that ${\rm Ch}(\Gamma_0, 2, \mathbb I)(\C)$ consists of two points. 
The first one $\mathbb L_1$ has representation
\ga{}{\rho_1(a)=\frac{1}{\sqrt{2}} \begin{pmatrix}
1 &  1 \\
-1 & 1 \notag
\end{pmatrix},
\ \ \rho_1(b)= \begin{pmatrix}
j &  0 \\
0 & j^2 
\end{pmatrix},
\notag} 
where $j$ is a primitive $3$-rd root of unity.  It is defined over $\Q(j)$. $\mathbb L_2$ is Gaois conjugate to $\mathbb L_1$. The authors compute
\ga{}{ \rho(ab) =\frac{j}{\sqrt{2}} \begin{pmatrix}
1 &  j \\
-1 & j  \notag
\end{pmatrix}. }
As ${\rm Trace}(\rho(ab))=-\frac{1}{\sqrt{2}}$, $\mathbb L_1$ is not integral  at  $\ell=2$, so $\mathbb L_2$ is not integral at  $\ell=2$ either. Furthermore, 
$\rho_1(a)$ does not preserve the eigenvalues of $\rho_1(b)$, so $\mathbb L_1$ and thus $\mathbb L_2$ are irreducible with dense monodromy in $\SL_2(\C)$.
They also compute that $H^1(\Gamma_0, \sE nd^0(\mathbb L_i))=0$.  That is,  in ${\rm Ch}(\Gamma_0,2,\mathbb I)$,  those two points are cohomologically rigid.


\begin{thebibliography}{BBDE04} 



\bibitem[AE19]{AE19} Abe, T., Esnault, H.: {\it A Lefschetz theorem for overconvergent isocrystals with Frobenius structure}, Annales de l'\'Ecole Normale Supérieure, {\bf 52} (4) (2019), 1243--1264. 


\bibitem[AB94]{AB94}
A'Campo, N., Burger, M.: {\it R\'eseaux arithmétiques et commensurateur d'après G. A. Margulis}, Invent. Math. {\bf 116} (1994), no. 1-3, 1–25.

\bibitem[BBV22]{BBV22} Becker, O., Breuillard, E., Varj\'u, P.: {\it Random character varieties}, 2022, in progress. 

\bibitem[Che14]{Che14}
Chenevier, G.: {\it The p-adic analytic space of pseudocharacters of a
profinite group and pseudorepresentations over arbitrary rings},
Automorphic forms and Galois representations. Vol. 1, 221–285,
London Math. Soc. Lecture Note Ser., {\bf 414},
Cambridge Univ. Press, Cambridge, 2014.

\bibitem[D'Ad20]{D'Ad20} D'Addezio, M.:
{\it The monodromy groups of lisse sheaves and overconvergent
F-isocrystals}, Selecta Math. (N.S.) {\bf 26} (2020), no. 3, Paper No. 45, 41 pp.

\bibitem[dJ01]{dJ01} de Jong, J.: {\it  A conjecture on the arithmetic fundamental group}, Israel J. of Math. {\bf 121} (2001), 61–84.

\bibitem[dJEG22]{dJEG22} de Jong, J., Esnault, H., Groechenig, M.: {\it Rigid non-cohomologically rigid local systems}, preprint 2022, 6 pages,
{\url{https://arxiv.org/pdf/2206.03590.pdf} v2}, to appear in Algebraic Geometry and Physics.

\bibitem[Del72]{Del72} Deligne, P.:{\it Les constantes des \'equations fonctionnelles des fonctions $L$}, Proc. Antwerpen Conference, vol. 2, Lecture Notes in Mathematics, {\bf 349}, pp. 501--597,  Springer Verlag, Berlin.
 
\bibitem[Del73]{Del73}  Deligne, P.: {\it Comparaison avec la th\'eorie transcendente}, in: SGA 7 II, Exp. No. XIV, pp. 116–164,
Lecture Notes in Mathematics, vol. {\bf 340}, (1973).

\bibitem[Del80]{Del80} Deligne, P.:  {\it La conjecture de Weil II}, Publ. math. I.H.\'E.S. {\bf 52} (1980), 137--252.

\bibitem[Del12]{Del12} Deligne, P.: {\it  Finitude de l’extension de $\Q$  engendr\'ee par des traces de Frobenius, en caract\'eristique finie,} Volume dedicated to the memory of I. M. Gelfand, Moscow Math. J. {\bf 12} (2012) no. 3.

\bibitem[Dri01]{Dri01} Drinfeld, V.: {\it  On a conjecture of Kashiwara}, Math. Res. Lett. {\bf 8} (2001), no. 5-6, 713--728.

\bibitem[Dri12]{Dri12} Drinfeld, V.: {\it On a conjecture of Deligne}, Volume dedicated to the memory of
I. M. Gelfand, Moscow Math. J. {\bf 12} (2012) no. 3.

\bibitem[DS05]{DS05} Dru\c{t}u, C., Sapir, M.: {\it Non-linear residually finite groups}, J. of Algebra {\bf 285} (2005),  174--178.

\bibitem[EG18]{EG18} Esnault, H., Groechenig, M.: {\it Cohomologically rigid connections and integrality}, Selecta Mathematica {\bf 24} (5) (2018), 4279---4292. 
\bibitem[EG20]{EG20} Esnault, H., Groechenig, M.: {\it Rigid connections and F-isocrystals}, Acta Mathematica {\bf 225} (1) (2020), 103---58. 
\bibitem[EK12]{EK12} Esnault, H., Kerz, M.: {\it A finiteness theorem for Galois representations of function fields over finite fields (after Deligne)}. Acta Mathematica Vietnamica {\bf 37} 4 (2012), 531--562.
\bibitem[EK22]{EK22} Esnault, H., Kerz, M.: {\it Density of Arithmetic Representations of Function Fields}, 18 pages, Epiga {\bf 6} (2022).
\bibitem[EK23]{EK23} Esnault, H., Kerz, M.: {\it Local systems with quasi-unipotent monodromy at infinity are dense}, preprint 2021, 9 pages 
\bibitem[Gai07]{Gai07} Gaitsgory, D.: {\it On de Jong's conjecture},  Israel J. Math. {\bf 157} (2007), 155--191.


\bibitem[Kli19]{Kli19} Klingler, B.: {\it $p$-adic lattives are not K\"ahler groups}, Epiga {\bf 3} (2019), Article 6. 
\bibitem[Laf02]{Laf02} Lafforgue, L.: {\it Chtoucas de Drinfeld et correspondance de Langlands}, Invent.
math. {\bf 147} (2002), no. 1, 1--241.

\bibitem[Lam91]{Lam91}
Lam, T.Y.: {\it A first course in noncommutative rings},
Graduate Texts in Mathematics, {\bf 131}.
Springer-Verlag, New York, 1991. xvi+397 pp. ISBN: 0-387-97523-3

 \bibitem[LL22a]{LL22a} Landesman, A., Litt, D.: {\it Geometric local systems on very general curves and isomonodromy}, 
 {\url{ https://arxiv.org/pdf/2202.00039.pdf} v2} .
\bibitem[LL22]{LL22} Landesman, A., Litt, D.:  {\it Canonical representations of surface groups},  
 {\url{https://arxiv.org/pdf/2205.15352.pdf}  v1}.
\bibitem[Maz89]{Maz89} 
Mazur, B.: {\it Deforming Galois Representations},  In: Galois groups over $\Q$ (Proc. Workshop, Berkeley/CA (USA), 1987), pp. 385--437, Mathematical Sciences Research Institute Publications, vol. {\bf 16}, Springer-Verlag, New York, 1989.
\bibitem[Moc07]{Moc07}
Mochizuki, T.: {\it Asymptotic behaviour of tame harmonic bundles and an
application to pure twistor D-modules. II.}
Mem. Amer. Math. Soc. {\bf 185} (2007), no. 870, xii+565 pp.

\bibitem[MF82]{MF82}
Mumford, D., Fogarty, J.: {\it Geometric invariant theory},
Second Edition. Ergebnisse der Mathematik und ihrer Grenzgebiete, {\bf 34}.
Springer-Verlag, Berlin, 1982. xii+220 pp.

\bibitem[Pet20]{Pet20} Petrov, A.: {\it Geometrically irreducible p-adic local systems are de Rham up to a twist}, to appear in Duke, {\url{https://arxiv.org/abs/2012.13372}}.

 \bibitem[ST68]{ST68} Serre, J.-P., Tate, J.: {\it Good Reduction of Abelian Varieties}, Annals of Math. {\bf 88} no 3 (1968), 492--517.

\bibitem[Sim92]{Sim92} Simpson, C.:  {\it Higgs bundles and local systems}, Publ. math. I.H.\'E.S. {\bf 75} (1992),
5--95.

\bibitem[WE18]{WE18}
Wang-Erickson, C., {\it Algebraic families of Galois representations and
potentially semi-stable pseudodeformation rings}, Math. Ann. {\bf 371} (2018), no. 3-4, 1615–1681.

\bibitem[SP]{SP} {\it The Stacks project},
\url{https://stacks.math.columbia.edu}.

\bibitem[SGA1]{SGA1}  {\it S\'eminaire de G\'eom\'etrie Alg\'ebrique Rev\^etements \'etales et groupe fondamental}, Lecture Notes in Mathematics {\bf 224},  Springer Verlag (1971).

\bibitem[SGA7.2]{SGA7.2} {\it S\'eminaire de G\'eom\'etrie Alg\'ebrique: Groupes de monodromie en g\'eom\'etrie  alg\'ebrique}, Lecture Notes in Mathematics {\bf 340}, Springer Verlag (1973).



 \end{thebibliography}
\end{document}